\documentclass[a4paper, final, notitlepage, 11pt]{article}
\usepackage{latexsym,bm}
\usepackage{amsthm}
\usepackage{amsfonts}
\usepackage{amsmath}
\usepackage[shortalphabetic,abbrev]{amsrefs}
\usepackage[all]{xy}
\usepackage[pdftex,final]{graphics}
\usepackage[dvipsnames]{pstricks}
\usepackage{verbatim}
\usepackage{geometry}
\usepackage{xcolor}
\newcommand{\e}{\varepsilon}
\newcommand{\nequiv}{\stackrel{num}{\sim}}

\newcommand{\B}{\tilde{B}}

\renewcommand{\O}{\mathcal{O}}
\renewcommand{\L}{\mathcal{L}}

\newcommand{\tpi}{\tilde{\pi}}

\newtheorem{theorem}{Theorem}[section]
\newtheorem{proposition}[theorem]{Proposition}
\newtheorem{corollary}[theorem]{Corollary}
\newtheorem{lemma}[theorem]{Lemma}
\theoremstyle{definition}

\theoremstyle{definition}
\newtheorem{remark}{Remark}[section]

\numberwithin{equation}{section}
\renewcommand{\labelenumi}{(\arabic{enumi})}
\begin{document}\title{\Large{\textbf{A New Family of Surfaces of General Type \\with $\mathbf{K^2=7}$ and $\mathbf{p_g=0}$}}}
\date{}
\author{\small{Yifan Chen}}
\maketitle
\renewcommand{\thefootnote}{\fnsymbol{footnote}}
\footnotetext{Yifan Chen: School of Mathematics Sciences, Peking University, Beijing 100871, P.~R.~China.}
\footnotetext{email address: chenyifan1984@gmail.com}

\begin{abstract}
We construct a new family of smooth minimal surfaces of general type with $K^2=7$ and $p_g=0.$
We show that for a surface in this family, its canonical divisor is ample and its bicanonical morphism is birational.
We also prove that these surfaces satisfy Bloch's conjecture.
\end{abstract}

\section{Introduction}
Minimal surfaces of general type with $p_g(S)=0$ have been constructed and studied
since the 1930's (cf.~\cite{Campedelli} and \cite{Godeaux}).
These surfaces have invariants $p_g(S)=q(S)=0$ and $1 \le K_S^2 \le 9.$
For each value of $K_S^2,$
except for the case $K_S^2=7,$
there exists nowadays quite a list of examples.
Up to the best knowledge of the author,
there is only one known family of minimal surfaces of general type
with $K^2=7$ and $p_g=0$ (cf.~\cite{survey} Tables 1-3).
This family of surfaces is due to M.~Inoue (cf.~\cite{Inoue}).
In \cite{bidoubleplane}, a family of surfaces of general type with $K^2=7$ and $p_g=0$ is constructed.
We will show in the last section that this family actually consists of Inoue surfaces.

Inoue surfaces with $K^2=7$ are constructed in \cite{Inoue}
as quotients of  complete intersections of codimension two in the product of four elliptic curves
by a fixed point free action.
Inoue surfaces can also be constructed as finite $(\mathbb{Z}/2\mathbb{Z})^2$-covers
of the $4$-nodal cubic surface (cf.~\cite{K27I} Example~4.1).
The bicanonical morphism of Inoue surfaces has degree $2$ and is composed with exactly one involution of $(\mathbb{Z}/2\mathbb{Z})^2.$
We refer to a recent article \cite{Inouemfd},
where the authors use both two constructions to study the deformations of Inoue surfaces and
generalize their results to certain manifolds.

In spite of lack of examples, there are many studies on minimal smooth surfaces of general type with $K^2=7$ and $p_g=0.$
It is shown in \cite{K27I} and \cite{K27II} that the bicanonical morphism of such a surface has degree either $1$ or $2.$
And if the bicanonical morphism has degree $2,$
the surface has a genus $3$ hyperelliptic fibration.
Involutions on surfaces of general type with $K^2=7$ and $p_g=0$ are studied in \cite{Lee} and \cite{bidoubleplane}.
Either article gives a list of numerical possibilities.
However, no new example is constructed (cf.~Section~6).
It is also shown in a pre-version of \cite{Lee} that three quotients of an Inoue surface
by the involutions are all rational.
However, we point out that one of the quotients is birational to an Enriques surface (cf. Section~6).

In this article, we construct a family of surfaces with $K^2=7$ and $p_g=0,$
as finite $(\mathbb{Z}/2\mathbb{Z})^2$-covers of certain weak Del Pezzo surfaces with degree one.
These surfaces have ample canonical divisors.

For a surface $S$ in our family,
we show that the bicanonical morphism of $S$ is not composed with any involution of $(\mathbb{Z}/2\mathbb{Z})^2.$
Indeed, by using the results of \cite{K27II},
we prove by contradiction that $S$ has birational bicanonical morphism.
So the family is indeed a new family.

We show that three quotients of $S$ by the involutions have respectively Kodaira dimensions $-\infty, 0, 1,$
realizing some numerical possibilities of the lists of \cite{bidoubleplane} and \cite{Lee}.
By applying the results of a recent article \cite{BlochInoue},
we prove that $S$ satisfies Bloch's conjecture.

\paragraph{Acknowledgement.}The author would like to thank Professor~Jinxing Cai
for many suggestions, for his patience and for supporting me many years.
The author would like to thank Ingrid~Bauer and Fabrizio~Catanese for interesting discussions on Inoue surfaces.
The author is very grateful to Yongnam~Lee, Wenfei~Liu, Carlos~Rito and Lei~Zhang for many discussions.

\section{Certain Weak Del Pezzo Surfaces of degree one}
We will construct a  family of weak Del Pezzo surfaces of degree one
as blowups of $\mathbb{P}^2$ at eight points.
We use $(x_1:x_2:x_3)$ as the homogeneous coordinates for $\mathbb{P}^2.$
Let $p_1=(1:0:0), p_2=(0:1:0), p_3=(0:0:1)$ and $p_0=(1:1:1),$
and let $p_j'$ be the infinitely near point over $p_j$ , corresponding to the line $\overline{p_jp_0}$, for $j=1, 2, 3.$
We state a lemma on conics passing some of these points.
\begin{lemma}\label{le!conic}
For each $i=1, 2, 3,$ there is a unique conic $c_i$ passing through $p_i,$ $p_{i+1},$ $p_{i+1}',$ $p_{i+2},$ $p_{i+2}'.$
Its equation is $x_i(x_{i+1}+x_{i+2})-x_{i+1}x_{i+2}=0.$
Moreover, $c_i$ does not pass through the point $p_i'.$
\end{lemma}
Here we make a convention that the indices $i \in \{1,2,3\}$ should be understood as residue classes modulo $3$
through the article.
We omit the proof of the lemma.

Let $\sigma \colon W \rightarrow \mathbb{P}^2$ be the blowup of eight points:
$p_0, p_1, p_1', p_2, p_2', p_3, p_3'$ and $p,$
where the eighth point $p$ satisfies:
\begin{enumerate}
\renewcommand{\labelenumi}{(\Roman{enumi})}
    \item $p \not \in \cup_{i=1}^3\{\overline{p_0p_i}:x_{i+1}=x_{i+2}\} \cup_{i=1}^3 \{ \overline{p_{i+1}p_{i+2}}: x_i=0\}.$
    \item $p \not \in c_1 \cup c_2 \cup c_3.$
\end{enumerate}
We remark that such surfaces $W$ are parameterized by $p.$

Denote by $E_j$ (respectively $E_j',$ $E$) the \textbf{total transform} of the point $p_j$ (respectively, $p_j',$ $p$),
and by $L$ the pullback of a general line by $\sigma.$ Then
$Pic(W)=\mathbb{Z}L\oplus \mathbb{Z}E_0\oplus \oplus_{j=1}^3 (\mathbb{Z}E_j\oplus \mathbb{Z}E_j')\oplus \mathbb{Z}E$
and $-K_W \equiv 3L-E_0-\sum\nolimits_{j=1}^{3}(E_j+E_j')-E.$
We list some properties of the surface $W.$
\begin{enumerate}
    \item $W$ is a weak Del Pezzo surface of degree $1,$
          i.e., $-K_W$ is nef and big, and $K_W^2=1.$

          This follows from the fact that any four points of $p_0, p_1, \ldots, p_3'$ and $p$ are not collinear
          (cf.~\cite{Topics}*{Theorem~8.1.7}).

    \item $W$ has exactly six $(-2)$-curves. Their divisor classes are as follows:
          \begin{equation}\label{eq:(-2)curves}
          C_j \equiv L-E_0-E_j-E_j',\  C_j'\equiv E_j-E_j',\ \text{for}\ j=1, 2, 3.
          \end{equation}

          Actually, assume that $C$ is a $(-2)$-curve of $W$
          and its divisor class is $C \equiv xL-a_0E_0-\sum_{j=1}^3 (a_jE_j+a_j'E_j')-aE.$
          If $\sigma(C)$ is a point, then $C$ is one of $C_1', C_2', C_3'.$
          If $c:=\sigma(C)$ is a curve, then $c$ is an irreducible curve of degree $x$
          having multiplicity at least $a_0$ (respectively $a_1, \ldots, a$) at the point $p_0$
          (respectively $p_1, \ldots, p$).
          In particular, $a_0, \ldots, a$ are nonnegative integers.
          If $x=1,$ then $C$ is one of $C_1, C_2, C_3.$

          It suffices to exclude the case $x \ge 2.$
          Since $C^2=-2$ and $K_WC=0,$
          \begin{align*}
          x^2+2=a_0^2+\sum_{j=1}^3(a_j^2+a_j'^2)+a^2,&&
          3x=a_0+\sum_{j=1}^3(a_j+a_j')+a.
          \end{align*}
          By Cauchy's inequality,
          $9x^2 \le (x^2+2) \cdot 8$ and thus $x \le 4.$

          If $x=4,$ then the equality holds, and $a_0=\ldots=a=2.$
          Then $CC_1'=x-a_0-a_1-a_1'=-2.$ Thus $C \ge C_1'.$
          This gives a contradiction and thus $x \not= 4.$

          Assume that $x=2.$
          Then $c$ is an irreducible smooth conic.
          So $a_0, \ldots, a \in \{0, 1\}.$
          Moreover, $a_0+\sum_{j=1}^3(a_j+a_j')+a=a_0^2+\sum_{j=1}^3(a_j^2+a_j'^2)+a^2=6.$
          So exactly six of $a_0, \ldots, a$ are $1.$
          Using Lemma~\ref{le!conic} and the condition (II),
          we see that there is no smooth conic  passing through six points of $p_0, \ldots, p.$
          Hence $x \not =2.$

          Assume that $x=3.$
          Then $c$ is an irreducible cubic curve.
          So $a_0, \ldots, a \in \{0, 1, 2\}.$
          Moreover, $a_0+\sum_{j=1}^3(a_j+a_j')+a=9$ and $a(a-1)+\sum_{j=1}^3(a_j(a_j-1)+a_j'(a_j'-1))+a(a-1)=2.$
          So exactly one of $a_0 ,\ldots, a$ is $2,$ and the others are $1.$
          If $a_0=2$ or $a_j=2$ or $a_j'=2$, then $CC_j=-1.$
          This gives a contradiction.
          So $a=2$ and $C \equiv -K_W-E.$
          It is more complicated to exclude this case.
          For later use, we state a lemma.

          \begin{lemma}\label{le!empty}$|-K_W-E|= \emptyset.$
          \end{lemma}
          \begin{proof}Assume by contradiction that $|-K_W-E| \not = \emptyset.$
          Then an element in $|-K_W-E|$ corresponds to a cubic curve $c$ on $\mathbb{P}^2$
          passing $p_0, p_1, p_2, p_3, p_1', p_2', p_3'$ and having a singularity at $p.$
          Let $F(x_1,x_2,x_3)$ be the equation of $c.$
          Since $c$ passes through $p_1, p_2, p_3,$ $F$ has no terms $x_1^3, x_2^3, x_3^3.$
          Since $\overline{p_0p_j}:x_{j+1}=x_{j+2} $ is the tangent line to $c$ at the point $p_j,$
          the coefficient of the term $x_j^2x_{j+1}$ is the opposite of that of the term $x_j^2x_{j+2}.$
          So we may assume that
          $$F(x_1,x_2,x_3)=Ax_1^2(x_2-x_3)+Bx_2^2(x_3-x_1)+Cx_3^2(x_2-x_1)+Dx_1x_2x_3.$$
          Since $c$ contains $p_0=(1:1:1),$ $D=0.$

          Assume that $p=(1:\alpha:\beta),$ where $\alpha \not=0, 1,$ and $\beta \not= 0, 1,$
          and $\alpha \not=\beta$ (cf.~(I)).
          The singularity $p$ of $c$ imposes the following conditions:
          \begin{align*}
          (\alpha-\beta)A+\alpha^2(\beta-1)B+\beta^2(\alpha-1)C=0,\\
          A+2\alpha(\beta-1)B+\beta^2C=0,\\
          -A+\alpha^2B+2\beta(\alpha-1)C=0.
          \end{align*}
          Since the coefficients matrix has determinant $2\alpha \beta(\alpha-1)(\beta-1)(\alpha-\beta),$
          which is nonzero, $A=B=C=0.$
          Hence $|-K_W-E| = \emptyset.$
          \end{proof}

    \item The $(-2)$-curves $C_1, C_1', C_2, C_2', C_3, C_3'$ are disjoint.
          Let $\eta \colon W \rightarrow \Sigma$ be the morphism contracting
          these $(-2)$-curves.
          Then $\Sigma$ has six nodes and $-K_\Sigma$ is ample.
    \item Denote by $\Gamma$ the strict transform of the line passing $p_0$ and $p,$
          i.e., $\Gamma \equiv L-E_0-E.$
          $\Gamma$ is a $(-1)$-curve and
          $\Gamma+E$ is  a member in a base-point-free pencil of rational curves $|F|,$
          where $F \equiv L-E_0.$
          $|F|$ corresponds to the pencil of lines on $\mathbb{P}^2$ passing through the point $p_0.$
          The morphism $g \colon W \rightarrow \mathbb{P}^1$ induced by $|F|$
          has exactly four singular fibers:
          $$\xymatrix @M=0pt@R=2pt{
          C_j              &   E_j'            &  C_j'   & \Gamma & E  \\
          \circ \ar@{-}[r] &  \circ \ar@{-}[r] &  \circ  & \circ \ar@{-}[r] &  \circ \\
          1&2&1&1&1}$$
          for $j=1, 2, 3.$
\end{enumerate}
From two $(-1)$-curves $\Gamma$ and $E,$
we will find two more $(-1)$-curves.
For this purpose, we need some properties of the linear system $|-2K_W|.$
\begin{proposition}[cf.~\cite{Topics}*{Theorem~8.3.2}]\label{prop!antibicanonical}
$h^0(W, \O_W(-2K_W))=4$ and $|-2K_W|$ defines a regular map
$\phi \colon W \rightarrow \mathbb{P}^3.$
It factors as a birational morphism $\eta \colon W \rightarrow \Sigma$
contracting exactly the six $(-2)$-curves,
and a finite morphism  $q \colon \Sigma \rightarrow Q$ of degree $2,$
where $Q$ is a quadric cone.
\end{proposition}
See \cite{Topics}*{Theorem~8.3.2} for a general statement on weak Del Pezzo surfaces of any degree and for a proof.

\begin{proposition}\label{prop!GammaE}
\begin{enumerate}
    \item The linear system of $|-2K_W-\Gamma|$ consists of a $(-1)$-curve.
          Denote this $(-1)$-curve by $\B_2.$
          Then $\B_2\Gamma=3$ and $\B_2E=1.$
    \item The linear system of $|-2K_W-E|$ consists of a $(-1)$-curve.
          Denote this $(-1)$-curve by $\B_3.$
          Then $\B_3\Gamma=1$ and $\B_3E=3.$
    \item $\Gamma+E+\B_2+\B_3$ has only nodes.
\end{enumerate}
\end{proposition}
\proof
\begin{enumerate}
\item  We have an exact sequence
       $$0 \rightarrow \O_W(-2K_W-\Gamma) \rightarrow \O_W(-2K_W) \rightarrow \O_{\Gamma}(-2K_W) \rightarrow 0.$$
       Since $-2K_W\Gamma=-2,$ by Proposition~\ref{prop!antibicanonical},
       $h^0(W, \O_W(-2K_W-\Gamma)) \ge 1.$

       Since $(-2K_W-\Gamma)^2=-1$ and $K_W(-2K_W-\Gamma)=-1,$
       it suffices to prove that any curve in $|-2K_W-\Gamma|$ is irreducible.

       First we prove that $|-K_W-\Gamma|$ is empty.
       Since $-K_W-\Gamma \equiv 2L-E_1-E_1'-E_2-E_2'-E_3-E_3',$
       this follows from Lemma~\ref{le!conic}.

       Assume that $-2K_W-\Gamma \equiv A_1+A_2,$
       where $A_1$ is an irreducible curve with $-K_WA_1=1,$
       and $Supp(A_2)$ is contained in the union of the $(-2)$-curves.
       By the algebraic index theorem, $A_1^2 \le 1.$
       If $A_1^2=1,$ then $A_1\equiv -K_W.$
       But then $-K_W-\Gamma \equiv A_2.$
       This contradicts that $|-K_W-\Gamma| =\emptyset.$
       So $A_1^2 <0$ and thus $A_1$ is a $(-1)$-curve.
       Since $Supp(A_2)$ is contained in union of the $(-2)$-curves,
       $A_1A_2+A_2^2=(-2K_W-\Gamma)A_2=0.$
       So $A_1A_2=-A_2^2$ is an even integer.
       Since $-K_W(A_1+A_2)=1,$ by the algebraic index theorem, $(A_1+A_2)^2 \le 1,$ i.e., $A_1A_2-1 \le 1.$
       If the equality holds, then $-K_W \equiv A_1+A_2.$
       Then $-K_W-\Gamma \equiv 0.$
       This contradicts that $|-K_W-\Gamma|=\emptyset.$
       It follows that $A_1A_2 =0.$
       So $A_2^2=0$ and thus $A_2=0.$

       Hence $|-2K_W-\Gamma|$ consists of a $(-1)$-curve $\B_2.$
       Moreover, $\B_2\Gamma=(-2K_W-\Gamma)\Gamma=3$ and $\B_2E=(-2K_W-\Gamma)E=1.$

\item  The proof is similar to (1).
       The key point is to prove that $|-K_W-E|=\emptyset.$
       This is true by Lemma~\ref{le!empty}.

\item  Recall that $\Gamma+E$ are disjoint from the $(-2)$-curves,
       since they are in different fibers of $g.$
       It follows that $\B_2 \equiv -2K_W-\Gamma$ and $\B_3 \equiv -2K_W-E$ are also disjoint from the
       $(-2)$-curves.
       Note that $\B_2\B_3=(-2K_W-\Gamma)(-2K_W-E)=1,$
       $\B_2E=\B_3\Gamma=\Gamma E=1.$
       It suffices to prove that
       \begin{enumerate}
       \item $\Gamma$ (respectively $E$) intersects $\B_2+\B_3$ transversely.
       \item $\B_2$ (respectively $\B_3$) intersects $\Gamma+E$ transversely.
       \end{enumerate}

       For (a), let $M:=\B_2+\B_3.$
       Then $|M|$ induces a genus $0$ fibration $h \colon W \rightarrow \mathbb{P}^1.$
       Since $MC_j=MC_j'=0$ for $j=1, 2, 3,$
       the six $(-2)$-curves are contained in the singular fibers of $h.$
       We claim that $h$ has exactly four singular fibers: $\B_2+\B_3$ and $M_j (j=1,2,3):$
       $$\xymatrix @M=0pt@R=2pt{
          (-2)             &  \Theta_j           & (-2)\\
          \circ \ar@{-}[r] &  \circ \ar@{-}[r] &  \circ  \\
          1&2&1}$$
       where the $(-2)$-curves are $C_1, \ldots, C_3',$
       and $\Theta_j$ is a $(-1)$-curve for $j=1, 2, 3.$

       Actually, since $-K_W$ is nef, for any irreducible component $A$ in a singular fiber,
       $A$ is either a $(-2)$-curve or $A$ is $(-1)$-curve.
       Since $-K_WM=-K_W(\B_2+\B_3)=2,$ any singular fiber contains either one $(-1)$-curve with multiplicity $2,$
       or two $(-1)$-curves with each multiplicity $1.$
       Since all $(-2)$-curves of $W$ are disjoint, any singular fiber has one of the following possible types:
       $$\xymatrix @M=0pt@R=2pt{
          (-2)             &  (-1)          & (-2)       &(-1) & (-2)& (-1) & (-1) & (-1)\\
          \circ \ar@{-}[r] &  \circ \ar@{-}[r] &  \circ  &\circ\ar@{-}[r] & \circ\ar@{-}[r] & \circ\ &
          \circ\ar@{-}[r] & \circ\\
          1&2&1&1&1&1&1&1}$$
       Each fiber of the first two types contributes $2$ to the Picard number $\rho(W).$
       Note that $W$ has six $(-2)$-curves $C_1, \ldots,C_3'$ and $\rho(W)=9.$
       By concerning how the $(-2)$-curves distribute to the singular fibers,
       we see that except the singular fiber $\B_2+\B_3,$
       any other singular fiber is of the first type.
       Our claim is proved.

       Since $M\Gamma=(-4K_W-\Gamma-E)\Gamma=4,$
       $h|_\Gamma \colon \Gamma \rightarrow \mathbb{P}^1$ is of degree $4.$
       Denote by $R$ the ramification divisor of $h|_\Gamma.$
       Since $\Gamma$ is disjoint from $C_j$ and $C_j',$
       $\Gamma \Theta_j=\frac 12 \Gamma M=2.$
       Thus $h|_\Gamma$ has ramification points on the singular fibers $M_1, M_2,M_3,$
       and $\deg R \ge 2 \times 3=6.$
       The Riemann-Hurwitz formula shows that $h|_\Gamma$ does not have any other ramification points
       than those on $M_1, M_2, M_3.$
       In particular, $\Gamma$ intersects the fiber $\B_2+\B_3$ transversely.

       Similar argument shows that $E$ intersects $\B_2+\B_3$ transversely.

       For (b), we use another fibration $g \colon W \rightarrow \mathbb{P}^1.$
       We have seen the singular fibers of $g$ in Section~2.
       Note that $\B_2F=\B_2(\Gamma+E)=4$ and $\B_3F=\B_3(\Gamma+E)=4.$
       Similar argument as the proof of (a) shows that $\B_2$ (respectively $\B_3$) intersects $\Gamma+E$ transversely. \qed
\end{enumerate}

\section{Construction of surfaces of general type}
In this section,
we construct a family of surfaces of general type as finite $(\mathbb{Z}/2\mathbb{Z})^2$-covers of $W.$
First, we define three effective divisors on $W$
\begin{align}
    \Delta_1:&=F_b+\Gamma+(C_1+C_1'+C_2+C_2') \equiv 4L-4E_0-2E_1'-2E_2'-E,\nonumber\\
    \Delta_2:&=\B_2+(C_3+C_3') \equiv  -2K_W-2E_3'+E, \label{eq:newdata1}\\ 
    \Delta_3:&=\B_3 \equiv -2K_W-E. \nonumber 
\end{align}
Here we require that
\begin{enumerate}
\renewcommand{\labelenumi}{(\Alph{enumi})}
  \item $F_b$ is a smooth fiber of $g$ (cf.~Section~2, property (4) of $W$).
  \item The divisor $\Delta:=\Delta_1+\Delta_2+\Delta_3$ has only nodes.

        By Proposition~\ref{prop!GammaE}, $\B_2+\B_3+\Gamma$ has only nodes,
        and $F_b, \Gamma, \B_2, \B_3$ are disjoint from the $(-2)$-curves $C_1, \ldots, C_3'.$
        $(B)$ is equivalent to that $F_b$ intersects $\B_2$ and $\B_3$ transversely,
        and does not pass the intersection point of $\B_2$ and $\B_3.$
        By Bertini theorem, this is the case for a general fiber.
\end{enumerate}

We also define three divisors
\begin{equation}\label{eq:newdata2}
    \begin{aligned}
     \L_1&=-2K_W-E_3',\\
     \L_2&=-K_W+(2L-2E_0-E_1'-E_2'-E),\\
     \L_3&=-K_W+(2L-2E_0-E_1'-E_2'-E_3'). 
    \end{aligned}
\end{equation}
It follows that $2\L_i \equiv \Delta_{i+1}+\Delta_{i+2},\ \ \L_i+\Delta_i \equiv \L_{i+1}+\L_{i+2}$
for all $i=1,2,3.$

Denote by $g_1,g_2,g_3$ the nonzero elements of $G:= (\mathbb{Z}/2\mathbb{Z})^2$
and by $\chi_i \in G^*$ the nontrivial character orthogonal to $g_i;$
by \cite{bidoublecover}*{Section~1} or \cite{singularbidouble}*{Theorem 2},
the data \eqref{eq:newdata1} and \eqref{eq:newdata2}
define a finite $G$-cover $\tpi \colon V \rightarrow W.$

By \cite{singularbidouble}*{Theorem~2}, conditions $(A)$ and $(B)$ imply that $V$ is smooth.
By the formulae in \cite{singularbidouble}*{Section~2},
\begin{align}
2K_V &\equiv \tpi^*(2K_W+\Delta) \equiv \tpi^*(-2K_W+\Gamma+\sum\nolimits_{j=1}^{3}(C_j+C_j')),\label{eq:newbican}\\
p_g(V) &=p_g(W)+\sum\nolimits_{i=1}^{3}h^0(W, \O_W(K_W+\L_i)).\label{eq:newpg}
\end{align}

Note that each $C_j$ or $C_j'$ (for $j=1,2,3$) is a connected component of $\Delta.$
The (set theoretic) inverse image $\tpi^{-1}C_j$ or $\tpi^{-1}C_j'$ is a disjoint union of two $(-1)$-curves.
Let $\e \colon V  \rightarrow S$ be the blowdown of these twelve $(-1)$-curves.
From the construction, there is a finite $G$-cover $\pi \colon S \rightarrow \Sigma$ such that the following diagram
commutes:
\begin{align}\label{diag}
\xymatrix{
    V \ar^{\e}"1,2" \ar_{\tpi}"2,1" & S \ar^{\pi}"2,2" \\
    W \ar^{\eta}"2,2" & \Sigma
}
\end{align}
The discussion above shows that
\begin{equation}
2K_S \equiv \tpi^*(-2K_\Sigma+\gamma), \label{eq:newbicanS}
\end{equation}
where $\gamma=\eta(\Gamma)$ is a $(-1)$-curve contained in the smooth part of $\Sigma.$

\begin{theorem}\label{thm!newsurface}
$S$ is a smooth minimal surface of general type with $K_S^2=7$ and $p_g(S)=0.$
Moreover, $K_S$ is ample.
\end{theorem}
\proof
   By \eqref{eq:newbicanS},
         $K_S^2=\frac 14 4(-2K_\Sigma+\gamma)^2=7.$

   By \eqref{eq:newpg} and $p_g(W)=0,$
   to show that $p_g(S)=p_g(V)=0,$ it suffices to show that $h^0(W, \O_W(K_W+\L_i))=0$ for $i=1, 2, 3.$
   By \eqref{eq:newdata2},
         \begin{equation}\label{eq:newKL}
         \begin{aligned}
         K_W+\L_1&\equiv -K_W-E_3',\\
         K_W+\L_2&\equiv 2L-2E_0-E_1'-E_2'-E,\\
         K_W+\L_3&\equiv 2L-2E_0-E_1'-E_2'-E_3'.
         \end{aligned}
         \end{equation}
   Recall the divisor classes of the $(-2)$-curves \eqref{eq:(-2)curves} for
   the calculation of intersection numbers.
   \begin{enumerate}
        \item   Assume by contradiction that $|K_W+\L_1| \not= \emptyset$
                and let $D \in |K_W+\L_1|.$
                Since $DC_3=DC_3'=-1,$  $D \ge C_3+C_3'.$
                Let $D':=D-C_3-C_3'.$
                Then $D'$ is effective and $D \equiv 2L-E_1-E_1'-E_2-E_2'-E_3-E.$
                But this contradicts Lemma~\ref{le!conic} and the condition (II).
                So $|K_W+\L_1|=\emptyset.$
        \item   Assume by contradiction that $|K_W+\L_2| \not= \emptyset$
                and let $D \in |K_W+\L_2|.$
                Since $DC_1=DC_1'=-1,$ $D \ge C_1+C_1'.$
                Then $(D-C_1-C_1')E_1'=-1$ and $D-C_1-C_1'\ge E_1'.$
                It follows that $D':=D-C_1-C_1'-E_1'$ is effective and $D' \equiv L-E_0-E_2'-E.$
                Similarly, $D'\ge C_2+C_2'+E_2'$ and $D'':=D'-C_2-C_2'-E_2'$ is effective.
                But $D'' \equiv -E.$
                This is impossible.
                So $|K_W+\L_2|=\emptyset.$

        \item   Assume by contradiction that $|K_W+\L_3| \not= \emptyset$
                and let $D \in |K_W+\L_3|.$
                Similar argument shows that $D \ge (C_1+C_1'+E_1')+(C_2+C_2'+E_2')$
                Thus $D':=D-(C_1+C_1'+E_1')-(C_2+C_2'+E_2')$ is effective.
                But $D' \equiv -E_3'.$
                This is impossible.
                So $|K_W+\L_3|=\emptyset.$
   \end{enumerate}
   Hence $p_g(S)=0.$

   Since $\pi$ is a finite morphism, by \eqref{eq:newbicanS},
   to prove that $K_S$ is ample,
   it suffices to show that $-2K_\Sigma+\gamma$ is ample,
   i.e., it suffices to show that if $C$ is an irreducible curve on $W$ such that $(-2K_W+\Gamma)C \le 0,$
   then $C$ is one of the six $(-2)$-curves $C_1, \ldots, C_3'.$

   Actually, since $-2K_W+\Gamma$ is effective and $(-2K_W+\Gamma)^2=7,$ $C^2 <0.$
   Since $-K_W$ is nef, $\Gamma.C \le 2K_W.C \le 0.$
   If $\Gamma.C <0,$ then $C=\Gamma$ and $C$ is a $(-1)$-curve.
   But then $-1=\Gamma.C \le 2K_WC=-2.$
   This gives a contradiction.
   Thus $\Gamma.C=0$ and $K_WC=0.$
   So $C$ is one of the $(-2)$-curves $C_1, \ldots,C_3'.$

   Hence  $K_S$ is ample and thus $S$ is minimal and of general type.\qed
\bigskip

We have constructed a family of surfaces with a $G \cong (\mathbb{Z}/2\mathbb{Z})^2$-action,
parameterized by a $3$-dimensional open subset $\{(p, F_b)| p \in \mathbb{P}^2\ \text{satisfying conditions (I) and (II)}, F_b\in |F|\ \text{satisfying conditions (A) and (B)}\}$
of $\mathbb{P}^2 \times \mathbb{P}^1.$
Here comes a natural question:
whether the family constructed here is new or not?
Up to our best knowledge,
the Inoue surfaces is the only one known family of surfaces with $K^2=7$ and $p_g=0$.
So we intend to show that the surfaces here have certain properties,
which are different from the Inoue surfaces.

\begin{proposition}$h^0(S, \O_S(2K_S))^{inv}=6, h^0(S, \O_S(2K_S))^{\chi_1}=h^0(S, \O_S(2K_S))^{\chi_2}=1$
and $h^0(S, \O_S(2K_S))^{\chi_3}=0.$
\end{proposition}
\proof
By the formulae in \cite{singularbidouble}*{Section~2}, for $i=1, 2, 3,$
$$h^0(S, \O_S(2K_S))^{\chi_i}=h^0(V, \O_V(2K_V))^{\chi_i}=h^0(W, \O_W(2K_W+\L_{i+1}+\L_{i+2}))$$
\begin{enumerate}
    \item  By \eqref{eq:newKL},
            \begin{align*}
            2K_W+\L_2+\L_3 &\equiv 4L-4E_0-2E_1'-2E_2'-E_3'-E\\
               &\equiv \Gamma+(C_1+C_1')+(C_2+C_2')+(C_3+E_3'+C_3').
            \end{align*}
            So $h^0(W, \O_W(2K_W+\L_2+\L_3)) \ge 1.$
            Let $D \in |2K_W+\L_2+\L_3|.$
            Then $DC_k=DC_k'=-2$ for $k=1,2$ and $DC_3=DC_3'=-1.$
            So $D\ge C_1+C_1'+C_2+C_2'+C_3+C_3'$ and thus $D':=D-(C_1+C_1'+C_2+C_2'+C_3+C_3')$ is effective.
            $D' \equiv \Gamma+E_3',$ and $D'\Gamma=D'E_3'=-1.$
            Thus $D'=\Gamma+E_3'$ and $D=\Gamma +(C_1+C_1')+(C_2+C_2')+(C_3+E_3'+C_3').$
            Hence $h^0(W, \O_W(2K_W+\L_2+\L_3)) = 1.$
    \item
            By \eqref{eq:newKL},
            \begin{align*}
            2K_W+\L_1+\L_3 &\equiv 5L-3E_0-E_1-2E_1'-E_2-2E_2'-E_3-3E_3'-E \\
               &\equiv (2L-E_1-E_2-E_3-E_3'-E)+\sum\nolimits_{j=1}^3 (C_j+C_j'),
            \end{align*}
            where $|2L-E_1-E_2-E_3-E_3'-E|$ consists of a $(-1)$-curve, which is the strict transform of
            the conic on $\mathbb{P}^2$ passing five points $p_1, p_2, p_3, p_3'$ and $p.$
            So $h^0(W, \O_W(2K_W+\L_1+\L_3)) \ge 1.$
            An similar argument shows that $h^0(W, \O_W(2K_W+\L_1+\L_3)) = 1.$
    \item By \eqref{eq:newKL},
           \begin{align*}
            2K_W+\L_1+\L_2 &\equiv 5L-3E_0-E_1-2E_1'-E_2-2E_2'-E_3-2E_3'-2E\\
               &\equiv (2L-E_1-E_2-E_3-2E)+\sum\nolimits_{j=1}^3 (C_j+C_j')
           \end{align*}
            Clearly, $|2L-E_1-E_2-E_3-2E|=|(L-E_1-E)+(L-E_2-E)-E_3|$ is empty.
            Arguing by contradiction as the proof of Theorem~\ref{thm!newsurface},
            it is easy to show that $h^0(W, \O_W(2K_W+\L_1+\L_2))=0.$
    \end{enumerate}
Hence $h^0(S, \O_S(2K_S))^{\chi_1}=h^0(S, \O_S(2K_S))^{\chi_2}=1$ and $h^0(S, \O_S(2K_S))^{\chi_3}=0.$
Since $h^0(S, \O_S(2K_S))=K_S^2+1=8,$ $h^0(S, \O_S(2K_S))^{inv}=6.$\qed

\begin{corollary}The bicanonical morphism $\varphi:=\varphi_{2K_S} \colon S \rightarrow \mathbb{P}^7$
is not composed with any involution $g_i,$ for $i=1, 2, 3.$
\end{corollary}

By the corollary, for a surface $S$ in our family,
the pair $(S,G)$ here is a different from any Inuoe surface with the $G$-action
(cf.~\cite{K27I}*{Example 4.1}).
So at least the construction of the pair $(S,G)$ is a new example.

\section{The bicanonical map}
It is known that for an Inoue surface, the bicanonical morphism has degree $2$ (cf.~\cite{K27I}*{Example 4.1}).
In this section, we prove that the our surfaces have birational bicanonical morphism.
Hence they are new surfaces.

\begin{theorem}\label{thm!birational}For a surface $S$ in Theorem~\ref{thm!newsurface},
the bicanonical morphism $\varphi \colon S \rightarrow \mathbb{P}^7$ is birational.
\end{theorem}

To prove the theorem,
first we study the images of the curves $\Gamma$ and $E$ on the surface $S$
(see the diagram \eqref{diag} in Section~3).
Let $e:=\eta(E)$ and $\gamma:=\eta(\Gamma).$
\begin{lemma}\label{le!gammae}
\begin{enumerate}
\item
$\pi^*(\gamma)=2\gamma',$ where $\gamma'$ is a smooth elliptic curve with $K_S\gamma'=1$ and $\gamma'^2=-1.$
\item
Let $e':=\pi^*(e).$
Then $e'$ is a smooth curve of genus $2$ with $K_Se'=6$ and $e'^2=-4.$
Moreover, $e'\gamma'=2.$
\end{enumerate}
\end{lemma}
\proof Note that $\Gamma$ and $E$ are disjoint from the $(-2)$-curves,
it is essentially the same to make a discussion on the the covering $\tpi \colon V \rightarrow W$ (see \eqref{diag}).

Note that $\Gamma$ is a component of the total branch divisor $\Delta,$
 and $\Gamma$ is disjoint from $F_b$ and $C_1, \ldots, C_3',$
and $\Gamma$ intersects $\B_2+\B_3$ transversely at four points (cf.~Proposition~\ref{prop!GammaE}).
So $\tpi^*(\Gamma)=2\Gamma',$
and $\tpi|_{\Gamma'} \colon \Gamma' \rightarrow \Gamma$ is a double cover of $\Gamma$
branched over  four different points.
Thus $\Gamma'$ is a smooth elliptic curve.
$\Gamma'^2=\frac 14 4\Gamma^2=-1$ and by \eqref{eq:newbican},
$K_V\Gamma'=\frac 14 4(-2K_W+\Gamma+\sum\nolimits_{j=1}^3(C_j+C_j'))\Gamma=1.$

Note that $E$ is not a component of  the branch divisor, and $E$ intersects the total branch divisor $\Delta$
transversely (cf.~Proposition~\ref{prop!GammaE}),
and $E\Delta_1=E\Delta_2=1,$ $E\Delta_3=3.$
So the restricted $(\mathbb{Z}/2\mathbb{Z})^2$-cover over $E$ is
a $(\mathbb{Z}/2\mathbb{Z})^2$-cover over $E \cong \mathbb{P}^1$
with the data (cf.~\cite{survey}*{Subsection 4.2, Proposition~4.19}):
$D_1=q_1, D_2=q_2, D_3=q_3+ q_4+q_5$ and $L_1=L_2=\O_{\mathbb{P}^1}(2), L_3=\O_{\mathbb{P}^1}(1),$
where $q_1, \ldots, q_5$ are different points of $E.$
By \cite{survey}*{Corollary~4.21},
$E':=\tpi^*(E)$ is irreducible. 
By \cite{singularbidouble}*{Theorem~2}, $E'$ is also smooth.
$E'^2=4 E^2=-4$ and by \eqref{eq:newbican}, $K_VE'=\frac 12 4(-2K_W+\Gamma+\sum\nolimits_{j=1}^3(C_j+C_j'))E=6.$
So $E'$ has genus $2.$

Moreover, $\Gamma'E'=\frac 12 4 \Gamma E=2.$
\qed
\bigskip

Now we are ready to prove Theorem~\ref{thm!birational}.

\proof[The proof of Theorem~\ref{thm!birational}]
Assume by contradiction that $\varphi$ is not birational.
By \cite{K27I} and \cite{K27II},
$\deg \varphi=2$ and $S$ has a genus $3$ hyperelliptic fibration $f \colon S \rightarrow \mathbb{P}^1.$
Moreover, $f$ has five double fibers.
Denote the general fiber of $f$ by $\Phi.$
Then $K_S\Phi=4.$

First we show that $\Phi\gamma'=0.$
By Lemma~\ref{le!gammae}, $K_S(\gamma'+\Phi)=5.$
By algebraic index theorem, $2\Phi \gamma'-1 \le \frac{5^2}{7}$
and thus $\Phi \gamma' \le 2.$
If $\Phi \gamma'=2,$ then $f|_{\gamma'} \colon \gamma' \rightarrow \mathbb{P}^1$ is a double cover.
Since $f$ has five double fibers, $f|_{\gamma'}$ has at least five ramification points on $\gamma'.$
Because $\gamma'$ is a smooth elliptic curve, this contradicts the Riemann-Hurwitz formula.
So $\Phi \gamma' <2.$
Since $f$ has double fibers, $\Phi \gamma'=0.$

Since $K_S(e'+\Phi)=10,$ by the algebraic index theorem,
$2\Phi e'-4 \le \frac {10^2}{7}$ and thus $\Phi e' \le 9.$
Since $f$  has double fibers, $\Phi e'$ is an even integer.

If $\Phi e'=8,$ then $f|_{e'} \colon e' \rightarrow \mathbb{P}^1$ is a morphism of degree $8.$
Denote by $R$ the ramification divisor of $f|_{e'}.$
Since $f$ has five double fibers, $\deg R \ge 4 \times 5=20.$
Since $e'$ is a smooth curve of genus $2,$ this contradicts the Riemann-Hurwitz formula.
The case $\Phi e'=6$ can be excluded by a similar argument.
So $\Phi e' \le 4.$

If $\Phi e'=4$ or $2,$ then the intersection number matrix of $K_S, \gamma', e'$ and $\Phi$ is
$$\left(\begin{array}{cccc}
        7 & 1 & 6 & 4\\
        1 & -1& 2 & 0\\
        6 & 2 & -4& 4\\
        4 & 0 & 4 & 0\end{array}\right)\ \text{or} \ \ \ \
        \left(\begin{array}{cccc}
        7 & 1 & 6 & 4\\
        1 & -1& 2 & 0\\
        6 & 2 & -4& 2\\
        4 & 0 & 2 & 0\end{array}\right).$$
Either matrix is nondegenerate.
This contradicts that $h^2(S, \mathbb{C})=3.$

Thus $\Phi e'=0.$
We have seen that $\Phi \gamma'=0.$
Since $(2\gamma'+e')^2=0,$
by Zariski's Lemma, $2\gamma'+e' \nequiv r\Phi$ for $r \in \mathbb{Q}.$
Since $K_S\gamma'=1,$ $K_Se'=6$ and $K_S\Phi=4,$
$r=2.$
By \cite{K27II}, $f$ has exactly one reducible fiber, which contains exactly two irreducible components.
So $\gamma'$ and $e'$ are the two irreducible components of this fiber.
Then $m\gamma'+ne' \equiv \Phi$ for some positive integers $m,n.$
This contradicts that $2\gamma'+e' \nequiv 2\Phi.$

Hence $\varphi$ is birational.

\section{The intermediate double covers and Bloch's conjecture}
From the construction,
we see that the automorphism group of the surface in Theorem~\ref{thm!newsurface} contains at least three involutions.
Involutions on surfaces of general type with $K^2=7$ and $p_g=0$ are studied in \cite{bidoubleplane} and \cite{Lee}.
Both articles give a list of numerical possibilities.
The surfaces constructed here realize some numerical possibilities of their lists.

\begin{proposition}\label{prop!intermediate}Let $S$ be a surface as in Theorem~\ref{thm!newsurface}.
\begin{enumerate}
\item The involution $g_1$ has $9$ isolated fixed points on $S,$ and $S/g_1$ is a rational surface.
\item The involution $g_2$ has $9$ isolated fixed points on $S,$ and $S/g_2$ is birational to an Enriques surface.
\item The involution $g_3$ has $7$ isolated fixed points on $S.$
      $S/g_3$ has Kodaira dimension $1,$ and $K_{S/g_3}$ is nef.
\end{enumerate}
\end{proposition}
\begin{proof}
\begin{enumerate}
\item Consider the intermediate double cover $\tpi_1 \colon V_1 \rightarrow W$
      of $\tpi \colon V \rightarrow W$ (cf.~\eqref{diag}) associated to the data $\Delta_2+\Delta_3 \equiv 2\L_1.$
      $V_1$ has exactly one node lying over the node of $\B_2+\B_3.$
      The (set theoretic) inverse image $\tpi_1^{-1}C_3$ or $\tpi_1^{-1}C_3'$ is a $(-1)$-curve,
      while the inverse image $\tpi_1^{-1}C_k$ or $\tpi_2^{-1}C_k'$ ($k=1, 2$) is two disjoint $(-2)$-curves.
      Contracting all these curves, we obtain the quotient surface $S/g_1.$
      From the construction (cf.~\eqref{diag}), $S/g_1$ has exactly $9$ nodes
      (the images of the node of $V_1$ and the $8$ $(-2)$-curves $\tpi_1^{-1}C_k, \tpi_1^{-1}C_k'$ ($k=1, 2$).
      Hence $g_1$ has $9$ isolated fixed points on $S.$

      To show that $S/g_1$ is rational, it suffices to show that $V_1$ is rational.
      As it is shown in the proof of Proposition~2.4~(3),
      $|M|=|\B_2+\B_3|$ gives a genus $0$ fibration $h \colon W \rightarrow \mathbb{P}^1.$
      For a general $M,$
      $M(\Delta_2+\Delta_3)=0.$
      So the pullback of $M$ by $\tpi_1$ is two disjoint smooth rational curves.
      Applying Stein factorization to the morphism $h \circ \tpi \colon V_1 \rightarrow \mathbb{P}^1,$
      we conclude that $V_1$ has a genus $0$ fibration.
      As a quotient of $V,$ $q(V_1)=0.$
      Hence $V_1$ is a rational surface and so is $S/g_1.$

\item Consider the intermediate double cover $\tpi_2 \colon V_2 \rightarrow W$
      associated to the data $\Delta_1+\Delta_3 \equiv 2\L_2.$
      $V_2$ has exactly $5$ nodes lying over the $5$ nodes of $F_b+\Gamma+\B_3.$
      Contracting the set theoretic inverse image of $\tpi_2^{-1}(C_j)$
      and $\tpi_2^{-1}(C_j')$ ($j=1, 2, 3$),
      we obtain $S/g_2.$
      It has $9$ nodes
      (the images of the $5$ nodes and the $4$ (-2)-curves $\tpi_2^{-1}C_3, \tpi_2^{-1}C_3'$).
      Hence $g_2$ has $9$ isolated fixed points on $S.$

      Clearly, $p_g(V_2)=q(V_2)=0.$ To show $V_2$ is birational to an Enriques surface.
      It suffices to show that $P_{2m+1}(V_2)=0$ and $P_{2m}(V_2)=1$ for $m \ge 1.$

      Note that $K_{V_2}=\tpi_2^*(K_W+\L_2).$
      So $P_{2m}(V_2)=h^0(W, \O_W(2mK_W+(2m-1)\L_2)+h^0(W, \O_W(2mK_W+2m\L_2)).$
      By \eqref{eq:newKL},
      \begin{align}\label{eq:P2m}
      2mK_W+2m\L_2 \equiv 2m\Gamma+m(C_1+C_1'+C_2+C_2').
      \end{align}
      Recall that $\Gamma$ is a $(-1)$-curve,
      $C_1, C_1', C_2, C_2'$ are $(-2)$-curves and all these curves are disjoint.
      So $h^0(W, \O_W(2mK_W+2m\L_2))=1.$

      By \eqref{eq:newKL}, clearly $|2K_W+\L_2| = \emptyset.$
      For $m \ge 2,$ by \eqref{eq:P2m},
      $$2mK_W+(2m-1)\L_2 \equiv 2(m-1)\Gamma+(m-1)(C_1+C_1'+C_2+C_2')+(2K_W+\L_2).$$
      Note that $\Gamma(2K_W+\L_2)=-2,$ $C_k(2K_W+\L_2)=C_k'(2K_W+\L_2)=-1 (k=1, 2).$
      If $|2mK_W+(2m-1)\L_2| \not =\emptyset,$
      then $2(m-1)\Gamma+(m-1)(C_1+C_1'+C_2+C_2')$ lies in the fixed part of this linear system.
      This contradicts that $|2K_W+\L_2|=\emptyset.$
      Hence $|2mK_W+(2m-1)\L_2|=\emptyset$ and $P_{2m}(V_2)=1.$

      An similar argument by using \eqref{eq:P2m} shows that $P_{2m+1}(V_2)=0$ for $m \ge 1.$
      Hence $V_2$ is birational to an Enriques surface.

\item
Consider the intermediate double cover $\tpi_3 \colon V_3 \rightarrow W$
related to the data $\Delta_1+\Delta_2 \equiv 2\L_3.$
$V_3$ exactly has $7$ nodes lying over the $7$ nodes of the curve $F_b+\Gamma+\B_2.$
Note that the (set theoretic) inverse image $\tpi_3^{-1}C_j$ or $\tpi_3^{-1}C_j'$ ($j=1, 2, 3$) is a $(-1)$-curve.
Contracting these $(-1)$-curves, we obtain $S/g_3.$
$S/g_3$ has $7$ nodes and $g_3$ has $7$ isolated fixed points on $S.$
By construction,
there are double covers $\pi_3 \colon S/g_3 \rightarrow \Sigma$ and $p_3 \colon S \rightarrow S/g_3$
such that the following diagram commutes.
$$\xymatrix@R=1em{
V \ar^{\e}"1,3" \ar"2,2" \ar_{\tpi}"3,1"& & S \ar_{p_3}"2,4" \ar@{-}_{\pi}"2,3" &\\
& V_3 \ar"2,4" \ar^{\tpi_3}"3,1" & \ar"3,3" & S/g_3 \ar^{\pi_3}"3,3"\\
W \ar_{\eta}"3,3" & & \Sigma &
}$$

By \eqref{eq:newKL},
$$2K_{V_3} \equiv \tpi_3^*(2K_W+2\L_3) \equiv \tpi_3^*(L-E_0+C_1+C_2+C_3+C_1'+C_2'+C_3').$$
As is shown in Section~2, $|L-E_0|$ gives a genus $0$ fibration $g \colon W \rightarrow \mathbb{P}^1,$
and all the $(-2)$-curves $C_1, \ldots, C_3'$ are contained in the fibers.
It induces a fibration on $g' \colon \Sigma \rightarrow \mathbb{P}^1.$
Denote the general fiber of $g'$ by $F'.$
From the diagram above, $2K_{S/g_3} \equiv \pi_3^*(F').$
Thus $K_{S/g_3}$ is nef and $K_{S/g_3}^2=0.$
Since $(L-E_0).(\Delta_1+\Delta_2)=4,$
$|2K_{S/g_3}|$ gives an elliptic fibration of $S/g_3.$
So $S/g_3$ has Kodaira dimension $1.$
\end{enumerate}
\end{proof}

\begin{remark}\label{re!intermediate}
We remark that (2) (respectively (3)) realize some numerical
possibilities of case a) (respectively case b)) of \cite{bidoubleplane}*{Theorem~4.1}.
(1), (2) and (3) realize respectively the following possible cases in the list of \cite{Lee}:
 \begin{enumerate}
 \item $k=9, K_W^2=-2,$ $W$ is a rational surface, and $B_0={\Gamma_0 \atop (3,0)}+{\Gamma_1 \atop (1,-2)}.$
 \item $k=9, K_W^2=-2,$ $W$ is birational to an Enriques surface, and $B_0={\Gamma_0 \atop (3,-2)}.$
 \item $k=7, K_W^2=0,$ $W$ is minimal proper elliptic, and  $B_0={\Gamma_0 \atop (2,-2)}.$
 \end{enumerate}
All these cases are different from the Inoue surfaces.
See \cite{Lee}*{Section~5} and Section~6.
\end{remark}

Recently, it is shown in \cite{BlochInoue} that
the Bloch's conjecture (\cite{Bloch}) holds for Inoue surfaces with $K^2=7$ and $p_g=0,$
by using the method of ``enough automorphisms" (\cite{enough1} and \cite{enough2}).
We observe that the key results in \cite{BlochInoue} also apply for our surfaces.

\begin{theorem}
Let $S$ be a surface as in Theorem~\ref{thm!newsurface}.
Then $S$ satisfies the Bloch's conjecture, i.e.,
the kernel $T(S)$ of the natural morphism $A_0^0(S) \rightarrow Alb(S)$ is trivial.
In particular, $A_0^0(S)=0.$
\end{theorem}
\begin{proof} The first statement follows directly from \cite{BlochInoue}*{Proposition~1.3, Corollary~1.5} and Proposition~\ref{prop!intermediate}.
Since in our case $Alb(S)$ is trivial, Bloch's conjecture says that $A_0^0(S)=0.$
\end{proof}

\section{Remarks on Related Topics}
In a previous version of \cite{Lee},
it was claimed in that three quotients of an Inoue surface by the involutions were all rational.
The claim turns out to be wrong.
We will point out that one of the quotient is birational to an Enriques surface.
In \cite{bidoubleplane},
a family of surfaces of general type with $K^2=7$ is constructed as bidouble planes.
However, here we show that the family with $K^2=7$ in \cite{bidoubleplane} consists of Inoue surfaces.

We first stick to the same notation with \cite{bidoubleplane}*{Section~4.2}.

Let $p, p_1, p_2, p_3$ be four points in general position of $\mathbb{P}^2,$
and let $p_k'$ ($k=1, 2$) be the infinitely near point of $p_k$ corresponding to the line $\overline{p_kp}.$
Denote by $T_j$ ( $j=1, 2, 3$) the line $\overline{p_jp}$ and by $T_4$ a general line passing through $p.$
Denote by $C_1, C_2$ two distinct  smooth conics passing through $p_1, p_2, p_1', p_2'.$
Denote by $L$ a quintic passing though $p,$ having a $(2, 2)$-singularity at $p_k$ ($k=1, 2$),
and having an ordinary triple point at $p_3$ (See the last paragraph in \cite{bidoubleplane}*{Subseciton~4.2.1}).

We claim that $L$ is a union of a conic $C$ and a cubic $\Gamma,$
where $C$ is the conic passing through $p_1, p_2, p_1', p_2', p_3,$
and $\Gamma$ is a cubic passing through $p, p_1, p_2, p_1', p_2'$ and having an ordinary double point at $p_3.$
Note that $L.C=11.$ The claim follows from B\'ezout's Theorem.

In \cite{bidoubleplane}, it is claimed that the smooth minimal model of the bidouble plane associated to the following
branch divisors is a surface of general type with $K^2=7$ and $p_g=0:$
\begin{align}
D_1=L=C+\Gamma,\
D_2=T_1+C_1+C_2,\
D_3=T_2+T_3+T_4. \label{eq:bidoubleplane}
\end{align}
We explain how to find the smooth minimal model of the bidouble plane,
and we show that this is indeed an Inoue surface with $K^2=7.$

\newcommand{\Y}{\widetilde{Y}}
\newcommand{\T}{\widetilde{T}}
\newcommand{\C}{\widetilde{C}}
Let $\sigma \colon \Y \rightarrow \mathbb{P}^2$ be the blowup of six point $p, p_1, p_2, p_3, p_1', p_2'.$
Denote by $L$ the pullback of a general line of $\mathbb{P}^2$
and by $E$ (respectively, $E_j$, $E_k'$) the total transform of $p$
(respectively, $p_j(j=1, 2, 3)$, $p_k'(k=1, 2)$).
We also denote by $\T_1$ the strict transform of $T_1,$ and similarly for other curves.

Then $\Y$ is the minimal resolution of a $4$-nodal cubic surface $Y.$
As is known, up to a (projective) isomorphism, there is only one $4$-nodal cubic surface.
We explain some geometry of $\Y.$
\begin{enumerate}
\item $\Y$ has exactly four $(-2)$-curves: $\T_1=L-E_1-E_1'-E,$ $\T_2=L-E_2-E_2'-E,$
      $N_1=E_1-E_1'$ and $N_2=E_2-E_2'.$
      These curves correspond to four nodes of $Y.$
\item $\Y$ contains nine $(-1)$-curves, corresponding to nine lines on the $4$-nodal cubic surface $Y.$
      Among these curves, there are exactly three, which are disjoint from the $(-2)$-curves:
      $\T_3=L-E_3-E,$
      $\C=2L-E_1-E_1'-E_2-E_2'-E_3$ and $E_3.$
      They correspond to three lines on $Y$ which do not pass any nodes.
      In particular, they are determined by the $4$-nodal cubic surface $Y.$
\item Note that
      \begin{align*}
      \T_4\in |L-E|, &&\T_4+\C \equiv -K_{\Y},\\
      \C_1,\C_2 \in|2L-E_1-E_1'-E_2-E_2'|,&& \C_1+\T_3\equiv \C_2+\T_3 \equiv -K_{\Y},\\
      \tilde{\Gamma}\in |3L-E_1-E_1'-E_2-E_2'-2E_3-E|,&& \tilde{\Gamma}+E_3 \equiv -K_{\Y}.
      \end{align*}
      So the divisor classes of $\T_4, \C_1, \C_2$ and $\tilde{\Gamma}$ are also determined by
      the $4$-nodal cubic surface.
\end{enumerate}

The total transforms of $D_1,D_2,D_3$ on $\Y$ are
\begin{align*}
\sigma^*(D_1)&=\C+\tilde{\Gamma}+2E_1+2E_1'+2E_2+2E_2'+3E_3+E,\\
\sigma^*(D_2)&=\T_1+\C_1+\C_2+N_1+2E_1'+E+2(E_1+E_1'+E_2+E_2'),\\
\sigma^*(D_3)&=\T_2+\T_3+\T_4+N_2+2E_2'+E_3+3E.
\end{align*}
Apply the normalization procedure in the theory of bidouble covers (cf.~\cite{singularbidouble}*{Section~2, Remark~3}),
we obtain three new divisors:
\begin{align}
\tilde{D}_1&=\C+\tilde{\Gamma}, \nonumber \\
\tilde{D}_2&=\T_1+\C_1+\C_2+N_1+E_3, \label{eq:normalization}\\
\tilde{D}_3&=\T_2+\T_3+\T_4+N_2.\nonumber
\end{align}

The bidouble cover $\tpi \colon \widetilde{S} \rightarrow \Y$ associated to \eqref{eq:normalization} is birational to the bidouble plane constructed by \eqref{eq:bidoubleplane}.
Using the above explanation of the geometry of $\Y$ (i.e., the $4$-nodal cubic surface),
and comparing \eqref{eq:normalization} with \cite{K27I}*{Example~4.1~(I)},
we conclude that the smooth minimal model of $\widetilde{S}$
(and thus of the bidouble plane) is an Inoue surface.
\bigskip

Now we point out a mistake in \cite{Lee}.
This observation is due to Carlos~Rito.
Here we  use the notation of \cite{K27I}*{Example~4.1},
as \cite{Lee} uses almost the same notation
(except denoting by $P$ the minimal resolution $\Sigma$ of the $4$-nodal cubic surface).
In \cite{Lee}*{Section 5, paragraph 4},
the author writes ``Also, $H^0(T_2, \O_{T_2}(2K_{T_2}))=0$ by a similar argument as the case $i=1$".
Here $T_2$ is a double cover of $\Sigma$ associated to $D_1+D_3 \equiv 2L_2.$
However,
we will show that $H^0(T_2, \O_{T_2}(2K_{T_2}))=1.$

It suffices to show $h^0(\Sigma, \O_{\Sigma}(2K_{\Sigma}+L_2))=0$ and $h^0(\Sigma, \O_{\Sigma}(2K_{\Sigma}+2L_2))=1,$
where $L_2=6l-2e_1-2e_2-2e_3-2e_4-3e_5-3e_6$ (\cite{K27I}*{Example~4.1}~(II)).
Since $2K_{\Sigma}+L_2=-e_5-e_6,$
clearly $h^0(\Sigma, \O_{\Sigma}(2K_{\Sigma}+L_2))=0.$
\begin{align*}
2K_{\Sigma}+2L_2&=6l-2e_1-2e_2-2e_3-2e_4-4e_5-4e_6\\
&\equiv (l-e_1-e_2-e_5)+(l-e_3-e_4-e_5)+(l-e_1-e_4-e_6)\\
&+(l-e_2-e_3-e_6)+2(l-e_5-e_6).
\end{align*}
Note that $l-e_1-e_2-e_5,$ $l-e_3-e_4-e_5,$ $l-e_1-e_4-e_6$ and $l-e_2-e_3-e_6$ are
$(-2)$-curves, and $(l-e_5-e_6)$ is a $(-1)$-curve (\cite{K27I}*{Figure~1}),
and all these curves are disjoint.
Hence $h^0(\O_{\Sigma}, \O_{\Sigma}(2K_{\Sigma}+2L_2))=1.$

Finally, as a comparison to Proposition~\ref{prop!intermediate} and Remark~\ref{re!intermediate},
we remark that $T_2$ is birational to an Enriques surface as described in \cite{bidoubleplane},
and it realizes the case $k=9, K_W^2=-2$ and $B_0={\Gamma_0 \atop (3,0)}+{\Gamma_1 \atop (1,-2)}$ in the list of \cite{Lee}.

\end{document}